\documentclass[a4paper,12pt]{amsart}
\usepackage{amsmath, amssymb, amsthm}
\usepackage{verbatim}
\usepackage{hyperref}

\newcounter{theorem}
\newtheorem{thm}[theorem]{Theorem}
\newtheorem{lemma}[theorem]{Lemma}

\newtheorem{cor}[theorem]{Corollary}
\newtheorem{defn}[theorem]{Definition}
\newtheorem{notation}[theorem]{Notation}

\theoremstyle{remark}
\newtheorem*{remark*}{Remark}

\numberwithin{equation}{section}
\numberwithin{theorem}{section}

\newcommand{\e}{\epsilon}

\renewcommand{\setminus}{\backslash}
\newcommand{\tens}{\otimes}

\newcommand{\dunion}{\amalg}

\renewcommand{\emptyset}{\varnothing}
\newcommand{\iso}{\cong}

\newcommand{\F}{\mathbb{F}}
\newcommand{\Q}{\mathbb{Q}}
\newcommand{\R}{\mathbb{R}}
\newcommand{\Z}{\mathbb{Z}}
\newcommand{\cS}{\mathcal{S}}

\newcommand{\ccite}[2]{\cite[#1]{#2}}
\newcommand{\alabel}{\label}

\newcommand{\labelledthing}[2]{\hspace{4pt}\buildrel {#2} \over #1 \hspace{3pt}} 

\newcommand{\labelledrightarrow}{\labelledthing{\longrightarrow}}

\begin{document}

\title[Ordered vector spaces with interpolation]{Finite dimensional ordered vector spaces with Riesz interpolation and Effros-Shen's unimodularity conjecture}

\author{Aaron Tikuisis}
\address{Aaron Tikuisis \\ \url{www.math.uni-muenster.de/u/aaron.tikuisis}}


\maketitle

\begin{abstract}
It is shown that, for any field $\F \subseteq \R$, any ordered vector space structure of $\F^n$ with Riesz interpolation is given by an inductive limit of sequence with finite stages $(\F^n,\F_{\geq 0}^n)$ (where $n$ does not change).
This relates to a conjecture of Effros and Shen, since disproven, which is given by the same statement, except with $\F$ replaced by the integers, $\Z$.
Indeed, it shows that although Effros and Shen's conjecture is false, it is true after tensoring with $\Q$.
\end{abstract}

\section{Introduction}
In this article we prove the following.

\begin{thm}
\alabel{MainResult}
Let $\F$ be a subfield of the real numbers, let $n$ be a natural number, and suppose that $(V,V^+)$ is a ordered directed $n$-dimensional vector space over $\F$ with Riesz interpolation.
Then there exists an inductive system
\[ (\F^n, \F_{\geq 0}^n) \labelledrightarrow{\phi_i^2} (\F^n, \F_{\geq 0}^n) \labelledrightarrow{\phi_2^3} \cdots \]
of ordered vector spaces over $\F$ whose inductive limit in $(V,V^+)$.
\end{thm}

The inductive limit may be taken either in the category of ordered abelian groups (with positivity-preserving homomorphisms as the arrows) or of ordered vector spaces over $\F$ (with positivity-preserving linear transformations as the arrows).
Here, $\F_{\geq 0} := \F \cap [0,\infty)$, so that the ordering on $(\F^n, \F_{\geq 0}^n)$ is simply given by coordinatewise comparison.

In \cite{EffrosShen:DimGroups}, Effros and Shen conjectured that every ordered, directed, unperforated, rank $n$ free abelian group $(G,G^+)$ with Riesz interpolation can be realized as an inductive system of ordered groups $(\Z^n, \Z_{\geq 0}^n)$. 
This was called the ``unimodularity conjecture,'' as the connecting maps would necessarily (eventually) be unimodular.
This conjecture was disproven by Riedel in \cite{Riedel}.
Theorem \ref{MainResult} shows that, nonetheless, upon tensoring with the rational numbers (or any other field contained in $\R$), the conjecture is true.
As a consequence, Corollary \ref{QSimplicial} says that if $(G,G^+)$ is an ordered $n$-dimensional $\Q$-vector space with Riesz interpolation then it is an inductive limit of $(\Z^n,\Z_{\geq 0}^n)$ (where the maps are, of course, not unimodular).

In \cite{Handelman:RealDG}, David Handelman shows that every vector space with Riesz interpolation can be realized as an inductive limit of ordered vector spaces $(\F^n,\F_{\geq 0}^n)$, though of course the number $n$ isn't assumed to be constant among the finite stages.
The focus of \cite{Handelman:RealDG} is on the infinite-dimensional case, and indeed, an interesting example is given of countable dimensional ordered vector space that can't be expressed as an inductive limit of a \textit{sequence} of ordered vector spaces $(\F^n,\F_{\geq 0}^n)$.
Combined with this article, this gives a dichotomy between the behaviour of infinite- versus finite-dimensional ordered vector spaces with Riesz interpolation.

\section{Preliminaries}
\alabel{PrelimSec}
We shall say a little here about the theory of ordered vector spaces with Riesz interpolation.
Although the focus in on vector spaces, much of the interesting theory holds in the more general setting of ordered abelian groups (particularly when the group is unperforated, as ordered vector spaces are automatically).
An excellent account of this theory can be found in the book \cite{Goodearl:book} by Ken Goodearl.

\begin{defn}
An \textbf{ordered vector space} consists of a vector space $V$ together with a subset $V^+ \subseteq V$ called the positive cone, giving an ordering compatible with the vector space structure; that is to say:
\begin{enumerate}
\item[(\textbf{OV1})] $V^+ \cap (-V^+) = 0$ ($V^+$ gives an order, not just a preorder);
\item[(\textbf{OV2})] $V^+ + V^+ \subseteq V^+$; and
\item[(\textbf{OV3})] $\lambda V^+ \subseteq V^+$ for all $\lambda \in \F_{\geq 0}$.
\end{enumerate}
The ordering on $V$ is of course given by $x \leq y$ if $y-x \in V^+$.

The ordered vector space $(V,V^+)$ is \textbf{directed} if for all $x,y \in V$, there exists $z \in V$ such that
\[ \begin{array}{c} x \\ y \end{array} \leq z. \]

The ordered vector space $(V,V^+)$ has \textbf{Riesz interpolation} if for any $a_1,a_2,c_1,c_2 \in V$ such that
\[ \begin{array}{c} a_1 \\ a_2 \end{array} \leq  \begin{array}{c} c_1 \\ c_2 \end{array}, \]
there exists $b \in V$ such that
\[ \begin{array}{c} a_1 \\ a_2 \end{array} \leq b \leq \begin{array}{c} c_1 \\ c_2 \end{array}. \]
\end{defn}

Note that $(V,V^+)$ being directed is an extremely natural condition, as it is equivalent to saying that $V^+ - V^+ = V$.
Riesz interpolation for an ordered vector space $(V,V^+)$ is equivalent to Riesz decomposition, which says that for any $x_1,x_2,y \in V^+$, if $y \leq x_1 + x_2$ then there exist $y_1,y_2 \in V^+$ such that $y = y_1+y_2$ and $y_i \leq x_i$ for $i=1,2$ \ccite{Section 23}{Birkhoff:LatticeGroups}.

The category of ordered vector spaces (over a fixed field $\F$) has as arrows linear transformations which are positivity-preserving, meaning that they map the positive cone of the domain into the positive cone of the codomain.
This category admits inductive limits, and for an inductive system $((V_\alpha, V_\alpha^+)_{\alpha \in A},(\phi_\alpha^\beta)_{\alpha \leq \beta})$, the inductive limit
is given concretely as $(V,V^+)$ where $V$ is the inductive limit of $((V_\alpha)_{\alpha \in A},(\phi_\alpha^\beta)_{\alpha \leq \beta})$
in the category of vector spaces, and if $\phi_\alpha^\infty:V_\alpha \to V$ denotes the canonical map then
\[ V^+ = \bigcup_{\alpha \in A} \phi_\alpha^\infty(V_\alpha). \]
If $(V_\alpha,V_\alpha^+)$ has Riesz interpolation for every $\alpha$ then so does the inductive limit $(V,V^+)$.

Theorem 1 of \cite{Handelman:RealDG} states that every ordered $\F$-vector space with Riesz interpolation can be realized as an inductive limit of a net of ordered vector spaces of the form $(\F^n,\F^n_{\geq 0})$.
The proof uses the techniques of \cite{EffrosHandelmanShen}, where it was shown that every ordered directed unperforated abelian group with Riesz interpolation is an inductive limit of a net of ordered groups of the form $(\Z, \Z_{\geq 0})$.
In the case that $\F=\Q$, \ccite{Theorem 1}{Handelman:RealDG} follows from \cite{EffrosHandelmanShen} and the theory of ordered group tensor product found in \cite{GoodearlHandelman:tens}.
Certainly, if $(V,V^+)$ is an ordered directed $\Q$-vector space with Riesz interpolation then it can be written as an inductive limit of $G_\alpha = (\Z^{n_\alpha}, \Z_{\geq 0}^{n_\alpha})$, and then we have
\[ (V,V^+) \iso (\Q,\Q^+) \tens_{\Z} (V,V^+) \iso \lim (\Q, \Q_{\geq 0}) \tens_{\Z} (\Z^{n_\alpha}, \Z_{\geq 0}^{n_\alpha}) \iso \lim (\Q^\alpha, \Q_{\geq 0}^{n_\alpha}). \]
But in the case of other fields, we no longer have $(V,V^+) \iso (V,V^+) \tens_{\Z} (\F,\F_{\geq 0})$ (indeed, $\F \tens_{\Z} \F \not\iso \F$).
Indeed, although in the countable case, the net of groups in \cite{EffrosHandelmanShen} can be chosen to be a sequence, not every countable dimensional ordered vector spaces with Riesz interpolation is the limit of a sequence of ordered vector spaces $(\F^n,\F_{\geq 0}^n)$.
Theorem 2 of \cite{Handelman:RealDG} characterizes when the net from \ccite{Theorem 1}{Handelman:RealDG} can be chosen to be a sequence: exactly when the positive cone is countably generated.

Using \cite{EffrosHandelmanShen}, one sees that an obviously sufficient condition for $(V,V^+)$ to be the limit of a sequence of ordered vector spaces of the form $(\F^n,\F_{\geq 0}^n)$ is that 
\begin{equation}
\alabel{DGtens}
(V,V^+) \iso (G,G^+) \tens_{\Z} (\F,\F_{\geq 0}).
\end{equation}
This is the case whenever $\F=\Q$.
\ccite{Proposition 5}{Handelman:RealDG} also shows that \eqref{DGtens} holds when $(V,V^+)$ is simple, since in this case, we can in fact take $(G,G^+)$ to be a rational vector space.
Also, \eqref{DGtens} holds in the finite rank case, as \ccite{Theorem 3.2 and Corollary 6.2}{FinRiesz} likewise show that we can take $(G,G^+)$ to be a rational vector space.
However, Theorem \ref{MainResult} improves on this result in the finite rank case, by showing that the finite stages have an even more special form -- their dimension does not exceed the dimension of the limit.

\section{Outline of the proof}
In light of the concrete description above of the inductive limit of ordered vector spaces, saying that $(V,V^+)$ (where $\dim_\F V = n$) can be realized as an inductive limit of a system
\[ (\F^n, \F_{\geq 0}^n) \labelledrightarrow{\phi_i^2} (\F^n, \F_{\geq 0}^n) \labelledrightarrow{\phi_2^3} \cdots \]
is equivalent to saying that there exist linear transformations $\phi_i^\infty:\F^n \to V$ such that:
\begin{enumerate}
\item $\phi_i^\infty$ is invertible for all $i$;
\item $V^+ = \bigcup \phi_i^\infty(\F_{\geq 0}^n)$; and
\item For all $i$, $\phi_i^\infty(\F_{\geq 0}^n) \subseteq \phi_{i+1}^\infty(\F_{\geq 0}^n)$.
\end{enumerate}
This idea is used in the proof of Theorem \ref{MainResult}, which we outline now.

We rely on \cite{FinRiesz} for a combinatorial description of the ordered vector space $(V,V^+)$.
Using this description, linear tranformations $\alpha^\e,\beta^R:\F^n \to \F^n$ are defined for all $\e,R \in \F_{>0} := \F \cap (0,\infty)$.
It is shown in Lemma \ref{Invertibility} that both $\alpha^\e$ and $\beta^R$ are invertible.
In \eqref{Udefn}, we associate to $(V,V^+)$ another ordered vector space $(\F^n,U^+)$ whose cone is like $V^+$ but such that the positive functionals on $(\F^n,U^+)$ separate the points.
We show in Lemma \ref{ForwardImagePositive} (i) and Lemma \ref{UnionIsAll} (i), that
\[ U^+ = \bigcup_{R \in \F_{>0}} \alpha^\e(\F_{\geq 0}^n), \]
and in Lemma \ref{ForwardImagePositive} (ii) and Lemma \ref{UnionIsAll} (ii), that
\[ V^+ = \bigcup_{\e \in \F_{>0}} \beta^R(U^+). \]
Although we don't have 
\[ \beta^{R_1}(\alpha^{\e_1}(\F_{\geq 0}^n)) \subseteq \beta^{R_2}(\alpha^{\e_2}(\F_{\geq 0}^n)) \]
 when $R_1 < R_2$ and $\e_1 > \e_2$, Lemma \ref{EpsilonExists} does allow us to extract an increasing sequence from among all the images $\beta^R(\alpha^\e(\F_{\geq 0}))$, such that their union is still all of $V^+$.

\section{The proof in detail}
We begin with a useful matrix inversion formula.

\begin{lemma}
\alabel{Jinverse}
Let $J_n \in M_n$ denote the matrix all of whose entries are $1$.
Then for $\lambda \neq -1/n$, $I_n+\lambda J_n$ is invertible and
\[ (I_n+\lambda J_n)^{-1} = I_n - \frac{\lambda}{\lambda n + 1} J_n. \]
\end{lemma}

\begin{proof}
Using the fact that $J^2 = nJ$, we can easily verify
\[ (I + \lambda J) \left(I - \frac{\lambda}{\lambda n + 1} J \right) = I. \]
\end{proof}

The main result of \cite{FinRiesz} shows that every finite dimensional ordered directed $\F$-vector space with Riesz interpolation looks like $\F^n$ with an positive cone given by unions of products of $\F,\F_{>0}$ and $\{0\}$.
To fully describe the result, the following notation for such products is quite useful.

\begin{notation}
For a partition $\{1,\dots,n\} = S_1 \dunion \cdots \dunion S_k$ and subsets $A_1,\dots,A_k$ of a set $A$, define
\[ A_1^{S_1} \cdots A_k^{S_k} = \{(a_1,\dots,a_n) \in A^n: a_i \in A_j \ \forall i \in S_j, j=1,\dots,k\}. \]
\end{notation}

\begin{thm} 
\alabel{CombDescr}
Every finite dimensional ordered directed $\F$-vector space with Riesz interpolation is isomorphic to $(\F^n,V^+)$ where
\[ V^+ = \bigcup_{S \in \cS} 0^{E^0_S}\, \F_{>0}^{E^>_{S}}\, \F^{E^*_{S}}, \]
$\cS$ is a sublattice of $2^{\{1,\dots,n\}}$ containing $\emptyset$ and $\{1,\dots,n\}$, and for each $S \in \cS$, we have a partition
\[ \{1,\dots,n\} = E^0_S \dunion E^>_S \dunion E^*_S, \]
where $E^0_S = S^c$.
We also use the notation $E^\geq_S := E^0_S \dunion E^>_S$.
The sets $E^0_S, E^>_S, E^*_S$ satisfy the following conditions:
\begin{enumerate}
\item[(\textbf{RV1})] \alabel{UnionFormula}
Using the notation $E^\geq_S := E^>_S \dunion S^c$, for any $S_1,S_2 \in \cS$,
\[ E^\geq_{S_1 \cup S_2} = E^\geq_{S_1} \cap E^\geq_{S_2}; \text{ and} \]
\item[(\textbf{RV2})] \alabel{PositiveIdealSeparation}
For any $S_1,S_2 \in \cS$, if $S_2 \not\subseteq S_1$ then $E^>_{S_2} \setminus S_1 \neq \emptyset$.
\end{enumerate}
\end{thm}

\begin{remark*}
Corollaries 5.2 and 6.2 of \cite{FinRiesz} says that, in the cases $\F=\R$ and $\F=\Q$, every $V^+$ given as in the above theorem does actually have Riesz interpolation.
Moreover, the proof of \ccite{Corollary 5.2}{FinRiesz} works for any other field $\F \subseteq \R$.
However, the proof of Theorem \ref{MainResult} only uses that $V^+$ has the form described in the above theorem, and therefore it gives an entirely different proof of \ccite{Corollary 5.2}{FinRiesz}, that $V^+$ has Riesz interpolation (since Riesz interpolation is preserved under taking inductive limits).
\end{remark*}

\begin{proof}
This is simply a special case of \ccite{Theorem 3.2}{FinRiesz}.
Note that (\textbf{RV2}) appears in \ccite{Theorem 3.2}{FinRiesz} as: if $S_1 \subsetneq S_2$ then $E^>_{S_2} \setminus S_1 \neq \emptyset$.
This is equivalent, to (\textbf{RV2}), since if $S_2 \not\subseteq S_1$ then $S_1 \cap S_2 \subsetneq S_2$, while if $S_1 \subsetneq S_2$ then of course $S_2 \not\subseteq S_1$.
\end{proof}

For each $i =1,\dots,n$, define
\begin{align*}
Z_i &:= \bigcup \{S \in \cS: i \in S^c\}, \text{ and} \\
P_i &:= \bigcup \{S \in \cS: i \in E^{\geq}_S\}.
\end{align*}
Note that $i \not\in Z_i^c$ and $i \in E^\geq_{P_i}$.

For $\e \in \F_{>0}$, define functionals $\alpha^\e_i:\F^n \to \F$ by
\begin{equation}
\alabel{alpha-defn}
 \alpha^\e_i(z_1,\dots,z_n) := z_i + \e \sum_{j \not\in Z_i} z_j;
\end{equation}
and for $R \in \F_{>0}$, define functionals $\beta^R_i:\F^n \to \F$ by
\begin{equation}
\alabel{beta-defn}
\beta^R_i(y_1,\dots,y_n) := y_i - R \sum_{j \not\in P_i, P_i \neq P_j} y_j.
\end{equation}

Let us denote $\alpha^\e := (\alpha^\e_1,\dots, \alpha^\e_n):\F^n \to \F^n$ and $\beta^R := (\beta^R_1,\dots,\beta^R_n):\F^n \to \F^n$.
Then $\alpha^\e$ is block-triangular, and $\beta^R$ is triangular, as we shall now explain.

For indices $i$ and $j$, we have $j \not\in Z_i$ if and only if $Z_i \subseteq Z_j$.
We therefore label the blocks of $(\alpha_1^{\e_1},\dots,\alpha_n^{\e_n})$ by sets $Z \in \cS$, where the $Z^\text{th}$ block consists of indices $i$ such that $Z_i = Z$; we shall use $B_Z$ to denote this set of indices, i.e.
\[ B_Z := \{i =1,\dots,n: Z_i = Z\}. \]

For $\beta^R$, note that if $i \not\in P_i$ then $P_i \subseteq P_j$, and from this it follows that $P_i$ is triangular.

\begin{lemma}
\alabel{Invertibility}
For all $\e,R \in \F_{>0}$, $\alpha^\e$ and $\beta^R$ are invertible.
\end{lemma}

\begin{proof}
That $\beta^R$ is invertible follows from the fact that it is triangular with $1$'s on the diagonal.
To show that $\alpha^\e$ is invertible, as we already noted that it is block-triangular, we need to check that each block is invertible.
In matrix form, the $Z^\text{th}$ block of $\alpha^R$ is equal to
\[ I_{|B_Z|} + \e J_{|B_Z|}, \]
and by Lemma \ref{Jinverse}, this block is invertible.
\end{proof}

\begin{notation}
\alabel{ZeroSet-Defn}
For $x \in \F^n$, let us use $S_x$ to denote the smallest set $S \in \cS$ such that $S$ contains
\[ \{i=1,\dots,n : x_i \neq 0\}. \]
\end{notation}

\begin{lemma}
\alabel{ZeroSet-Lemma}
Let $\e,R \in \F_{>0}$ be scalars and let $z \in \F^n$.
Then 
\[ S_z = S_{\alpha^\e(z)} = S_{\beta^R(\alpha^\e(z))}. \]
\end{lemma}

\begin{proof}
To show that $S_{\alpha^\e(z)} \subseteq S_z$, it suffices to show that $\alpha^\e_i(z)=0$ for all $i \not\in S_z$, which we show in (a).
Likewise we show in (b) that $z_i = 0$ for all $i \not\in S_{\alpha^\e(z)}$, in (c) that $\beta^R_i(\alpha^\e(z)) = 0$ for all $i \not\in S_{\alpha^\e(z)}$, and in (d) that $\alpha_i^\e(z) = 0$ for all $i \not\in S_{\beta^R(\alpha^\e(z))}$.

(a)
If $i \not\in S_z$ then $S_z \subseteq Z_i$ and therefore, for every $j \not\in Z_i$ we have $j \not\in S_z$ and so $z_i = 0$.
Since $\alpha^R_i(z)$ is a linear combination of $\{z_j: j \not\in Z_i\}$, it follows that $\alpha^R_i(z) = 0$.

(b)
We shall prove this by induction on the blocks $B_Z$, iterating $Z \in \cS$ in a nonincreasing order.
Since $i \not\in S_{\alpha^R(z)}$ if and only if $Z_i \supseteq S_{\alpha^R(z)}$, we only need to consider $Z \supseteq S_1$.

For a block $Z \supseteq S_{\alpha^R(z)}$ and an index $i \in B_Z$, we have
\begin{equation}
\alabel{ZeroSet-Eqb}
0 = \alpha^\e_i(z) 
= z_i + \e \sum_{j \in B_Z} z_j + \e \sum_{j: Z_j \supsetneq Z} z_j.
\end{equation}
By induction, we have that $z_j = 0$ for all $j$ satisfying $Z_j \subsetneq Z$; that is to say, the last term in \eqref{ZeroSet-Eqb} vanishes.
Hence, the system \eqref{ZeroSet-Eqb} becomes
\[
0 = (I_{|B_Z|} + \e J_{|B_Z|})(z_i)_{i \in B_Z};
\]
and by Lemma \ref{Jinverse}, it follows that $z_i = 0$ for all $i \in B_Z$, as required.

(c)
For (c) and (d), let us set $y := \alpha^\e(x)$.
If $i \not\in S_y$ then again, $S_y \subseteq Z_i$ and so $y_i = 0$ for all $j \not\in Z_i \subseteq P_i$.
Since $\beta^R_i(y)$ is a linear combination of $\{y_i\} \cup \{y_j: j \not\in P_i\}$, $\beta^R_i(y) = 0$.

(d)
If $i \not\in S_{\beta^R(y)}$ then we have
\[
0 = \beta^R_i(y)
= y_i - R\sum_{j \not\in P_i, P_i \supsetneq P_j} y_j.
\]
As above, $j \not\in P_i$ implies that $j \not\in S_{\beta^R(y)}$.
Hence, if we iterate the indices $i \in S_{\beta^R(y)}^c$ in a nondecreasing order of the sets $P_i$ then induction proves $y_i=0$ for all $i \not\in S_{\beta^R(y)}$.
\end{proof}

Our proof makes use of the following positive cone:
\begin{equation}
\alabel{Udefn}
 U^+ := \bigcup_{S \in \cS} \F_{>}^S\, 0^{S^c}.
\end{equation}

\begin{lemma}
\alabel{ForwardImagePositive}
Let $R,\e \in \F_{>0}$ be scalars.
Then:
\begin{enumerate}
\item $\alpha^\e(\F_{\geq 0}^n) \subseteq U_+$, and
\item $\beta^R(U_+) \subseteq V_+$.
\end{enumerate}
\end{lemma}

\begin{proof}
(i)
Let $z \in \F_{\geq 0}^n$.
By Lemma \ref{ZeroSet-Lemma}, we know that $\alpha^{\e_{Z_i}}_i(z)=0$ for $i \not\in S_z$.
Let us show that $\alpha^{\e_{Z_i}}_i(z) > 0$ for $i \in S_z$, from which it follows that $\alpha^\e(z) \in U^+$.

For $i \in S_z$, we have
\[ \alpha^{\e_{Z_i}}_i(z) = z_i + \e \sum_{j \not\in Z_i} z_j; \]
so evidently $\alpha^{\e_{Z_i}}_i(z) \geq 0$ and $\alpha^{\e_{Z_i}}_i(z) = 0$ would imply that $z_j = 0$ for all $j \not\in Z_i$.
But if that were the case, then we would have $S_z \subseteq Z_i$, and in particular, $i \not\in S_z$, which is a contradiction.
Hence $\alpha^{\e_{Z_i}}_i(z) > 0$.

(ii)
Let $y \in U^+$.
Then we must have $y_i > 0$ for all $i \in S_y$.
By Lemma \ref{ZeroSet-Lemma}, we already know that $\beta^R_i(y) = 0$ for all $i \in S_y^c = E^0_{S_y}$.
Thus, we need only show that $\beta^R_i(y) > 0$ for $i \in E^>_{S_y}$.
For such an $i$, we have
\[ \beta^R_i(y) = y_i - R\sum_{j \not\in P_i, P_j \neq P_i} y_j. \]
By (i), we know that $y_i > 0$.
Since $i \in E^>_{S_y}$, we have $P_i \subseteq S_y$.
Therefore if $j \not\in P_i$ then $j \not\in S_y$ and so $y_j = 0$.
Thus, we in fact have $\beta^R_i(y) = y_i > 0$.
\end{proof}

\begin{lemma}
\alabel{UnionIsAll}
Let $U^+$ be as defined in \eqref{Udefn}.
Then:
\begin{enumerate}
\item $U^+ = \bigcup_{\e \in \F_{>0}} \bigcap_{\e' \in \F_{>0}, \e' < \e} \alpha^{(\e'_Z)}(\F_{\geq 0}^n).$
\item $V^+ = \bigcup_{R \in \F_{>0}} \bigcap_{R' \in \F, R' > R} \beta^{R'}(U^+).$
\end{enumerate}
\end{lemma}

\begin{proof}
(i)
Let $y \in U^+$.
Define $m := \min \{|y_i|: i \in S_y\} > 0$ and $M := \max \{|y_i|: i \in S_y\}$, and suppose that $\e \in \F_{>0}$ is such that
\[ \e < \frac{m}{2nM}, \]
for all $Z \in \S$.
Let us show that $z = (\alpha^{\e})^{-1}(y)$ satisfies $z_i \geq 0$ for all $i$.

We will show, by induction on the blocks $B_Z$ (iterating $Z \in \cS$ in a nonincreasing order), that
\[ 0 \leq z_i \leq M \]
for all $i \in B_Z$.
By the definition of $S_y$, we already know that this holds for $Z \supseteq S_y$ (for if $i \not\in S_y$ then $z_i = 0$ for $i \not\in S_y$).

For $i \in B_Z \cap S_y$, set
\[ C_i := z_i + \e \sum_{j \in B_Z} z_j = y_i - \e \sum_{j \not\in Z, Z_j \supsetneq Z} z_j. \]

Then we have
\[ C_i \geq m - \e n M > m - m/2 = m/2 \]
and
\[ C_i \leq M. \]

By Lemma \ref{Jinverse}, we have
\[ z_i = C_i - \frac{\e}{n\e+1} \sum_{j \in B_Z} C_j. \]
On the one hand, this gives
\[ z_i > m/2 - \e nM = m/2 - m/2 = 0, \]
and on the other, it gives
\[ z_i \leq C_i \leq M, \]
as required.

(ii)
Let $x \in V^+$.
For $R \in \F_{>0}$ let us denote $y^R = (y^R_1,\dots,y^R_n) := (\beta^R)^{-1}(x)$.
For all $i \not\in S_x$ we already know that $y^R_i = 0$ for all $i$.
Moreover, for all $i \in E^>_{S_x}$ and all $R$, we have
\[ x_i = y^R_i - R \sum_{j \not\in P_i, P_j \neq P_i} y^R_i; \]
but note that if $j \not\in P_i \supseteq S_x$ then $j \not\in S_x$, and therefore we have $y^R_i = x_i > 0$.

We will show by induction that, for each $i \in E^*_{S_x}$ there exists $R_i \in \F_{>0}$ such that for all $R'' \geq R' \geq R_i$, we have
\[ y_i^{R''} > y_i^{R'} > 0. \]
We iterate the indices $i$ in a nonincreasing order of $P_i$.

For the index $i$, we have
\begin{equation}
\alabel{UnionIsAll-yEq}
y^R_i = x_i + R \sum_{j \not\in P_i, P_j \supsetneq P_i} y^R_j.
\end{equation}
If we require that $R \geq \max\{R_j: P_j \supsetneq P_i\}$ then, by induction, we know that $y^R_j \geq 0$ for all $j \not\in P_i$.
Moreover, since $i \not\in E^{\leq}_{S_x}$, this means that $S_x \not\subseteq P_i$ and therefore by (\textbf{RV2}) in Theorem \ref{CombDescr}, there exists some $j_0 \in E^>_{S_x} \setminus P_i$.
Notice that $P_j \subseteq S_x$ so that $y^R_j$ does appear as a summand in the right-hand side of \eqref{UnionIsAll-yEq}.
Thus, we have
\[
y^R_i = x_i + R \sum_j \not\in P_i, P_j \neq P_i y^R_j \geq x_i + R y^R_{j_0} = x_i + R x_j.
\]
Since $x_j > 0$, there exists $R=R_i$ for which the right-hand side is positive, and so $y^R_i > 0$.

Since $y^R_j$ is a nondecreasing function of $R$ for all $j$ for which $P_j \supsetneq P_i$, it is clear from \eqref{UnionIsAll-yEq} that so is $y^R_i$.
\end{proof}

\begin{lemma}
\alabel{EpsilonExists}
Let $R_1,\e_1 \in \F_{>0}$ be scalars.
For any $R' > R_1$, there exist $R_2,\e_2 \in \F_{>0}$ with $R_2 > R'$ and $\e_2 < \e_1$ such that
\[ \beta^{R_1}(\alpha^{\e_1}(\F_{\geq 0}^n)) \subseteq \beta^{R_2}(\alpha^{\e_2}(\F_{\geq 0}^n)). \]
\end{lemma}

\begin{proof}
Let $e_1,\dots,e_n$ be the canonical basis for $\F^n$, so that $\F_{\geq 0}^n$ is the cone generated by $e_1,\dots,e_n$.
Then for each of $i=1,\dots,n$, we have by Lemma \ref{ForwardImagePositive} that
\[ \beta^{R_1}(\alpha^{\e_1}(e_i)) \in V_+; \]
and thus by Lemma \ref{UnionIsAll} (i), there exists $R_2 > R'$ such that
\[ (\beta^{R_2})^{-1}(\beta^{R_1}(\alpha^{\e_1}(e_i))) \in U_+ \]
for all $i=1,\dots,n$.
By Lemma \ref{UnionIsAll} (ii), there then exists $\e_2 < \e_1$ such that
\[ (\alpha^{\e_1})^{-1}((\beta^{R_2})^{-1}(\beta^{R_1}(\alpha^{\e_1}(e_i)))) \in \F_{\geq 0}^n \]
for all $i=1,\dots,n$, which is to say,
\[ \beta^{R_1}(\alpha^{\e_1}(e_i)) \in \beta^{R_2}(\alpha^{\e_2}(\F_{\geq 0}^n)). \]
Since $\F_{\geq 0}^n$ is the cone generated by $e_1,\dots,e_n$, it follows that
\[ \beta^{R_1}(\alpha^{\e_1}(\F_{\geq 0}^n)) \subseteq \beta^{R_2}(\alpha^{\e_2}(\F_{\geq 0}^n)), \]
as required.
\end{proof}

\begin{proof}[Proof of Theorem \ref{MainResult}]
Let $R_1,\e_1 \in \F_{>0}$, and, using Lemma \ref{EpsilonExists}, inductively construct sequences $(R_i), (\e_i) \subset \F_{>0}$, such that $R_i \to \infty, \e_i \to 0$ and for each $i$,
\[ \beta^{R_i}(\alpha^{\e_i}(\F_{\geq 0}^n)) \subseteq \beta^{R_{i+1}}(\alpha^{\e_{i+1}}(\F_{\geq 0}^n)). \]
Set $\phi_i = \beta^{R_i} \circ \alpha^{\e_i}:\F^n \to \F^n$.
By Lemma \ref{UnionIsAll}, we have $V^+ = \bigcup_{i=1}^\infty \phi_i(\F_{\geq 0}^n)$.

Our inductive system is thus 
\[ (\F^n,\F_{\geq 0}^n) \labelledrightarrow{\phi_{2}^{-1} \circ \phi_1} (\F_n,\F_{\geq 0}^n) \labelledrightarrow{\phi_3^{-1} \circ \phi_2} \cdots; \]
as explained in Section \ref{PrelimSec}, the inductive limit is
\[ (\F^n, \bigcup_{i=1}^\infty \phi_i(\F_{\geq 0}^n)) = (V,V^+), \]
as required.
\end{proof}

\section{Consequences}

\begin{cor}
\alabel{QSimplicial}
Let $(V,V^+)$ be an $n$-dimensional ordered directed $\Q$-vector space with Riesz interpolation.
Then there exists an inductive system of ordered groups
\[ (\Z^n,\Z_{\geq 0}^n) \labelledrightarrow{\phi_1^2} (\Z^n,\Z_{\geq 0}^n) \labelledrightarrow{\phi_2^3} \cdots \]
whose inductive limit is $(V,V^+)$.
\end{cor}

\begin{proof}
By Theorem ref{MainResult}, let
\[ (\Q^n, \Q_{\geq 0}^n) \labelledrightarrow{\phi_i^2} (\Q^n, \Q_{\geq 0}^n) \labelledrightarrow{\phi_2^3} \cdots \]
be an inductive system whose limit is $(V,V^+)$.
Since positive scalar multiplication gives an isomorphism of any ordered vector space, we may replace any of the connecting maps with a positive scalar multiples, and still get $(V,V^+)$ in the limit.
Hence, we may assume without loss of generality that $\phi_i^{i+1}(\Z^n) \subseteq \Z^n$.
Then, letting $\overline{\phi}_i^{i+1} = \phi_i^{i+1}|_{\Z^n}$, we have an inductive system
\[ (\Z^n,\Z_{\geq 0}^n) \labelledrightarrow{\overline{\phi}_i^2} (\Z^n, \Z_{\geq 0}^n) \labelledrightarrow{\overline{\phi}_2^3} \cdots, \]
whose limit $(G,G^+)$ satisfies $(G,G^+) \tens_\Z (\Q,\Q^+) \iso (V,V^+)$.

Now, we may easily find an inductive system 
\[ (\Z,\Z_{\geq 0}) \labelledrightarrow{\psi_1^2} (\Z,\Z_{\geq 0}) \labelledrightarrow{\psi_2^3} \cdots \]
whose limit is $(\Q,\Q_{\geq 0})$.
(Such an inductive system necessarily has $\psi_i^{i+1}$ given by multiplication by a positive scalar $N_i$; and the limit is $(\Q,\Q_{\geq 0})$ as long as every prime occurs as a root of infinitely many $N_i$.)

Thus, by \ccite{Lemma 2.2}{GoodearlHandelman:tens}, $(V,V^+)$ is the inductive limit of
\[ (\Z^n,\Z_{\geq 0}^n) \tens_\Z (\Z,Z_{\geq 0}) \labelledrightarrow{\overline{\phi}_1^2 \tens_\Z \psi_1^2} (\Z^n, \Z_{\geq 0}^n) \tens_\Z (\Z,\Z_{\geq 0}) \labelledrightarrow{\overline{\phi}_2^3 \tens_\Z \psi_2^3} \cdots, \]
which is what we require, since $(G,G^+) \tens_\Z (\Z,\Z_{\geq 0}) = (G,G^+)$ for any ordered abelian group $(G,G^+)$.
\end{proof}

\begin{cor}
Let $(G,G^+)$ be a rank $n$ ordered directed free abelian group with Riesz interpolation.
Then there exists an inductive system of ordered groups
\[ (\Z^n,\Z_{\geq 0}^n) \labelledrightarrow{\phi_1^2} (\Z^n,\Z_{\geq 0}^n) \labelledrightarrow{\phi_2^3} \cdots \]
whose inductive limit is $(G,G^+) \tens_{\Z} (\Q,\Q_{\geq 0})$.
\end{cor}

\begin{proof}
This follows immediately, as $(G,G^+) \tens_{\Z} (\Q,\Q_{\geq 0})$ is an $n$-dimensional ordered directed $\Q$-vector space with Riesz interpolation.
\end{proof}

\end{document}